\documentclass[11pt]{amsart}
\usepackage{amsmath}
\usepackage{amssymb}
\usepackage{amscd}

\def\NZQ{\mathbb}               
\def\NN{{\NZQ N}}

\def\ZZ{{\NZQ Z}}

\newtheorem{Theorem}{Theorem}[section]
\newtheorem{Lemma}[Theorem]{Lemma}

\newtheorem{Proposition}[Theorem]{Proposition}

\newtheorem{Question}[Theorem]{Question}
\let\epsilon\varepsilon
\let\phi=\varphi
\let\kappa=\varkappa

\begin{document}

\title{Extensions of valuations to the Henselization and completion}
\author{Steven Dale Cutkosky}
\thanks{partially supported by NSF grant DMS-1700046}

\address{Steven Dale Cutkosky, Department of Mathematics,
University of Missouri, Columbia, MO 65211, USA}
\email{cutkoskys@missouri.edu}


\maketitle

\section{Introduction}

Suppose that $K$ is a field. Associated to a valuation $\nu$ of $K$  is a value group $\Phi_{\nu}$ and  a valuation ring $V_{\nu}$ with maximal ideal $m_{\nu}$. Let $R$ be a local domain with quotient field $K$. We say that $\nu$ dominates $R$ if $R\subset V_{\nu}$ and $m_{\nu}\cap R=m_R$ where $m_R$ is the maximal ideal of $R$. We have an associated semigroup 
$$
S^R(\nu)=\{\nu(f)\mid f\in R\setminus(0)\},
$$
 as well as the associated graded ring of $R$ along $\nu$
\begin{equation}\label{eqN31}
{\rm gr}_{\nu}(R)=\bigoplus_{\gamma\in \Phi_{\nu}}\mathcal P_{\gamma}(R)/\mathcal P^+_{\gamma}(R)=\bigoplus_{\gamma\in S^{R}(\nu)}\mathcal P_{\gamma}(R)/\mathcal P^+_{\gamma}(R)
\end{equation}
which is defined by Teissier in \cite{T1}. Here 
$$
\mathcal P_{\gamma}(R)=\{f\in R\mid \nu(f)\ge \gamma\}\mbox{ and }\mathcal P^+_{\gamma}(R)=\{f\in R\mid \nu(f)> \gamma\}.
$$ 
This ring plays an important role in local uniformization of singularities (\cite{T1} and \cite{T2}).
The ring ${\rm gr}_{\nu}(R)$ is a domain, but it is often not Noetherian, even when $R$ is.  In fact, a necessary condition for ${\rm gr}_{\nu}(R)$ to be Noetherian is that $\Phi_{\nu}$ be a finitely generated group.

In this paper, we answer the following question, which is a natural generalization of  local uniformization.

\begin{Question}\label{Question1} Suppose that $R$ is a Noetherian  local domain which is dominated by a valuation $\nu$. Does there exist a regular local ring $R'$ of the quotient field $K$ of $R$ such that $\nu$ dominates $R'$ and $R'$ dominates $R$, a prime ideal $p$ of the $m_{R'}$-adic completion $\widehat{R'}$ of $R'$  such that $p'\cap R'=(0)$ and  an extension $\hat\nu$ of $\nu$ to 
the quotient field of $\widehat{R'}/p$ which dominates $\widehat{R'}/p$ such that
$$
{\rm gr}_{\nu}(R')\cong {\rm gr}_{\hat\nu}(\widehat{R'}/p)?
$$
\end{Question}

 A nonzero prime ideal $p$ may be necessary to obtain the conclusions of Question \ref{Question1}, as is shown in \cite{S} and \cite{HS}.

 If $\nu$ has rank 1,  then we easily obtain a prime $p$ in the completion of $R$ such that  
$$
{\rm gr}_{\nu}(R)\cong {\rm gr}_{\hat\nu}(\hat R/p)
$$
 as we now indicate.

Suppose that $R$ is a Noetherian local domain which is dominated by a rank 1 valuation $\nu$. For $f\in \hat R$, we write $\nu(f)=\infty$ if there exists a Cauchy sequence $\{f_n\}$ in $R$ which converges to $f$, and such that
$\lim_{n\rightarrow \infty}\nu(f_n)=\infty$.
We define (\cite[Definition 5.2]{CG}) a prime ideal
$$
P(\hat R)_{\infty}=\{f\in \hat R\mid \nu(f)=\infty\}
$$
in $\hat R$. We then have a canonical immediate extension $\hat\nu$ of $\nu$ to ${\rm QF}(\hat R/P(\hat R)_{\infty})$ which dominates $\hat R/P(\hat R)_{\infty}$.
 
 The following lemma appears in \cite{GAST}.

\begin{Lemma}\label{Lemma12} Suppose that $\nu$ is  a rank 1 valuation of a field $K$ and $R$ is a Noetherian local domain which is dominated by $\nu$. Let $\hat\nu$ be the canonical extension of $\nu$ to ${\rm QF}(\hat R/P(\hat R)_{\infty})$ which dominates $\hat R/P(\hat R)_{\infty}$. Then the inclusion $R\rightarrow \hat R/P(\hat R)_{\infty}$ induces an isomorphism 
$$
{\rm gr}_{\nu}(R)\cong {\rm gr}_{\hat\nu}(\hat R/P(\hat R)_{\infty}).
$$
\end{Lemma}

\begin{proof}  Suppose $h\in \hat R\setminus P(\hat R)_{\infty}$. There exists a Cauchy sequence $\{f_n\}$ in $R$ such that
$\lim_{n\rightarrow \infty}f_n=h$. Let $m$ be a positive integer such that $m\nu(m_R)>\hat\nu(h)$ (where $\nu(m_R)=\min\{\nu(g)\mid g\in m_R\}$). There exists $n_0$ such that $f_n - h \in m_R^m\hat R$ for $n\ge n_0$. Then ${\rm in }_{\nu}(f_n)={\rm in }_{\hat \nu}(h)$ for $n\ge n_0$.
\end{proof}

From Lemma \ref{Lemma12}, we obtain a positive answer to Question \ref{Question1} for local domains $R$  and rank 1 valuations $\nu$ which admit local unformization.
A positive answer to Question \ref{Question1} for rank 1 valuations with the additional conclusion that $\widehat{R'}/p$ is a regular local ring is given in \cite{C2} and in  \cite[Theorem 7.2]{CG} for $R$ which are essentially of finite type over a field of characteristic zero.  This is generalized somewhat in \cite{CE} and \cite{Eh}.

Related to Question \ref{Question1} is the following question, which we will also answer.

\begin{Question}\label{Question2} Suppose that $R$ is a Noetherian local domain which is dominated by a valuation $\nu$. Does there exist a regular local ring $R'$ of the quotient field $K$ of $R$
 such that $\nu$ dominates $R'$ and $R'$ dominates $R$, and  an extension $\nu^h$ of $\nu$ to 
the quotient field of the Henselization $(R')^h$ of  $R'$ which dominates $(R')^h$ such that
$$
{\rm gr}_{\nu}(R')\cong {\rm gr}_{\nu^h}((R')^h))?
$$
\end{Question}

A positive answer to Question \ref{Question1} would imply a positive answer to Question \ref{Question2}.

We prove the following proposition  on the extension of associated graded rings under an unramified extension in \cite{C1}, which gives a start on answering Question \ref{Question2}.  Related problems are considered in \cite{GAST}.

\begin{Proposition}\label{PropUnR}(\cite[Proposition 1.7]{C1})  Suppose that $R$ and $S$ are  normal local rings such that $R$ is excellent,  $S$ lies over  $R$ and $S$ is unramified over $R$, $\tilde \nu$ is a valuation of the quotient field $L$ of $S$ which dominates $S$, and $\nu$ is the restriction of $\tilde \nu$ to the quotient field $K$ of $R$. Suppose that $L$ is finite  over $K$.  Then there exists a normal local ring $R'$ of $K$ which is dominated by $\nu$ and dominates $R$,
  such that if $R''$ is a normal local ring of $K$ which is dominated by $\nu$ and dominates $R'$
  and $S''$ is the normal local ring of $L$ which is dominated by $\tilde\nu$ and lies over $R''$, then $R''\rightarrow S''$ is unramified, and
$$
{\rm gr}_{\tilde\nu}(S'')\cong {\rm gr}_{\nu}(R'')\otimes_{R''/m_{R''}}S''/m_{S''}.
$$
\end{Proposition}

We give an example at the end of 
  of \cite[Section 5]{C1}  showing  that it may be necessary  to take $R'\ne R$ to obtain the conclusions of Proposition \ref{PropUnR} if $\nu$ has rank greater than 1. The ring $R$ is regular and there is no residue field extension in the example. This example shows that it may be necessary to perform a proper extension $R\rightarrow R'$ take $R'\ne R$ to obtain a positive answer to  Question \ref{Question2} or  \ref{Question1}, even if $R$ is a regular local ring. This problem arises from the fact that the residue field under blowing up of the center of a composite valuation can increase. 
 Related examples are considered in \cite{GAST}. In  \cite[Remark 2]{GAST}, it is already observed  that the increase of residue field under blowing up of the center of a composite valuation is a critical issue  in understanding an extension of a valuation dominating a local domain to its completion.
 
 In this paper, we show  that Question \ref{Question2} and Question \ref{Question1} have a negative answer in general.   This is accomplished in Theorems \ref{Theorem1} and \ref{Theorem2} stated below, which are proven in Section \ref{SecCounter}. The examples of the theorems are on  three dimensional regular local rings  which are a localization at a maximal ideal of a polynomial ring  over an arbitrary algebraically closed field.

 \begin{Theorem}\label{Theorem1}  Suppose $k$ is an algebraically closed field. Then there exists a three dimensional regular local ring $T_0$, which is a localization of a finite type $k$-algebra, with residue field $k$, and a valuation $\phi$ of the quotient field $K$ of $T_0$ which dominates $T_0$ and whose residue field is $k$, such that if $T$ is a regular local ring of $K$ which is dominated by $\phi$ and dominates $T_0$, $T^h$ is the Henselization of $T$ and $\phi^h$ is an extension of $\phi$ to the quotient field of $T^h$ which dominates $T^h$, then $S^{T^h}(\phi^h)\ne S^T(\phi)$,
under the natural inclusion $S^T(\phi)\subset S^{T^h}(\phi^h)$.
\end{Theorem}

\begin{Theorem}\label{Theorem2}  Suppose $k$ is an algebraically closed field. Then there exists a three dimensional regular local ring $T_0$, which is a localization of a finite type $k$-algebra, with residue field $k$, and a valuation $\phi$ of the quotient field $K$ of $T_0$ which dominates $T_0$ and whose residue field is $k$, such that if $T$ is a regular local ring of $K$ which is dominated by $\phi$ and dominates $T_0$
 and $\hat T$ is the $m_T$-adic completion of $T$, then there does not exist a prime ideal $p$ of $\hat T$ such that $p\cap T=(0)$ with an extension $\hat \phi$ to the quotient field of $\hat T/p$ which dominates $\hat T/p$ such that $S^{\hat T/p}(\hat \phi)=S^T(\phi)$, under the natural inclusion $S^T(\phi)\subset S^{\hat T/p}(\hat \phi)$.
\end{Theorem}

A  very interesting related problem, which is still open, is  \cite[Conjecture 1.1]{GAST}  on the existence of ``scalewise birational'' extensions of associated graded rings.   \cite[Conjecture 1.1]{GAST} is a refinement of  \cite[Statement 5.19]{T1}. The two results Theorem \ref{Theorem1} and Theorem \ref{Theorem2} of this paper are counterexamples to possible hopes of improving the statements of Conjecture 1.11 and Theorem 7.1 of \cite{GAST}.

This paper relies on the construction of generating sequences (Section \ref{SecGen}), using the algorithm of \cite{CV}, which is a generalization of the algorithm of \cite{S}. The construction of generating sequences in a local domain which is  dominated by a valuation which provide enough information to determine the associated graded ring along the valuation is an important problem. Some recent papers addressing this are \cite{Eh}, \cite{K}, \cite{Mo}  and \cite{NS}.
 
 \section{Notation}

The nonnegative integers will be denoted by $\NN$ and the positive integers will be denoted by $\ZZ_+$. If $\Lambda$ is a subset of an Abelian  group $G$ then $G(\Lambda)$ will denote the group generated by $\Lambda$ and $S(\Lambda)$ will denote the semigroup (containing zero) generated by $\Lambda$. 

The maximal ideal of a local ring $R$ will be denoted by $m_R$. Suppose that $K$ is a field. A local ring of $K$ is a local domain whose quotient field is $K$. We will say that a local domain $B$ dominates a local domain $A$ if $A\subset B$ and $m_B\cap A=m_A$. If a regular local ring $B$ dominates a local domain $A$ and $B$ is a local ring of the blow up of an ideal $I$ in $A$, then a strict transform of an element $f\in A$ in $B$ is a generator $g$ of the principal ideal $(f):_BIB$; that is, $g$ is a generator of the ideal of the strict transform of $\mbox{Spec}(A/(f))$ in $\mbox{Spec}(B)$.

If $\nu$ is a valuation of a field $K$, $V_{\nu}$ will denote the valuation ring of $\nu$, $m_{\nu}$ will denote the maximal ideal of $V_{\nu}$ and $\Phi_{\nu}$ will denote the value group of $\nu$.
The basics of valuation theory are explained in   \cite[Chapter VI]{ZS2} and   \cite[Chapter II]{RTM}. We will say that a valuation $\nu$ dominates a local domain $R$ if $V_{\nu}$ dominates $R$. We define the semigroup
$$
S^R(\nu)=\{\nu(f)\mid f\in R\setminus \{0\}\}.
$$

\section{Construction of generating sequences of a valuation}\label{SecGen}
Suppose that $R$ is a local domain and $\nu$ is a valuation dominating $R$. A sequence of elements $\{P_i\}$ in $R$ is called a generating sequence for $\nu$ if 
the classes ${\rm in}_{\nu}(P_i)$ of the $P_i$ in $\mathcal P_{\nu(P_i)}/\mathcal P_{\nu(P_i)}^+$ generate ${\rm gr}_{\nu}(R)$ as an $R/m_R$-algebra.

Let $k$ be an algebraically closed field, The regular local ring $R_0=k[x,y,t]_{(x,y)}\cong k(t)[x,y]_{(x,y)}$ has regular parameters $x,y$ and residue field $R_0/m_{R_0}\cong k(t)$. We will inductively define a generating sequence 
$$
P_0=x, P_1=y,P_2,\ldots
$$
in $R_0$,  defining a valuation $\nu$ dominating $R_0$, using the method of the proof of \cite[Theorem 1.1]{CV}. The generating sequence will satisfy the good conditions of the conclusions of \cite[Theorem 4.2]{CV}.  Let $\overline p_1,\overline p_2,\ldots$ be the sequence of prime numbers, excluding the characteristic of $k$. Define $a_1=\overline p_1+1$ and inductively define positive integers $a_i$ by
$$
a_{i+1}=\overline p_i\overline p_{i+1}a_i+1.
$$
Define 
$$
P_{i+1}=P_i^{\overline p_i^2}-(1+t)x^{\overline p_ia_i}
$$
for $i\ge 1$. Set $\nu(x)=1$ and 
$$
\nu(P_i)=\frac{a_i}{\overline p_i}\nu(x)\mbox{ for }i\ge 1.
$$
We have $\mbox{gcd}(a_i,\overline p_i)=1$ for all $i$.  Thus $G(\nu(P_0),\nu(P_1))=\frac{1}{\overline p_1}\ZZ$ and 
$$
\overline n_i:= [G(\nu(P_0),\ldots,\nu(P_i)):G(\nu(P_0),\ldots,\nu(P_{i-1})]=\overline p_i
$$
for $i\ge 1$. 

We have $\nu(P_i)>\overline p_{i-1}^2\nu(P_{i-1})$ for $i\ge 2$ and the value group
\begin{equation}\label{eq60}
\Phi_{\nu}=\cup_{i\ge 1}\frac{1}{\overline p_1\overline p_2\cdots\overline p_i}\ZZ.
\end{equation}
Let $\kappa$ be an algebraic closure of $k(t)$ and let $\alpha_i\in \kappa$ be a root of $f_i(u)=u^{\overline p_i}-(1+t)\in k[u]$ for $i\ge 1$. Since the $\overline p_i$ are all pairwise relatively prime, by induction on $i$, we see that $f_i(u)$ is the minimal polynomial of $\alpha_i$ over $k(\alpha_1,\ldots,\alpha_{i-1})$ and 
$$
d_i:= \overline p_i=[k(t)(\alpha_1,\ldots,\alpha_i):k(t)(\alpha_1,\ldots,\alpha_{i-1})]
$$
 for $i\ge 1$.   Since $n_i=d_i\overline n_i=\overline p_i^2$ for $i\ge 1$, by the algorithm of the proof of  \cite[Theorem 1.1]{CV}, the $P_i$ are the generating sequence of a valuation $\nu$ dominating $R_0$ which has the property that setting $U_i=x^{a_i}$ for $i\ge 1$, we have 
 $V_{\nu}/m_{\nu}=k(\{\alpha_i\mid i\ge 1\})$ with 
 $$
 \alpha_i=\left[\frac{P_i^{\overline p_i}}{U_i}\right]\in V_{\nu}/m_{\nu}
 $$
 for $i\ge 1$.

From $\overline p_1\nu(y)=a_1\nu(x)$ with $\mbox{gcd}(a_1,\overline p_1)=1$, there exist $\overline a_0,\overline b_0\in \ZZ_+$ such that $\overline p_1\overline b_0-a_1\overline a_0=1$. Set
$$
x=x_1^{\overline p_1}\tilde y_1^{\overline a_0},\,\, y=x_1^{a_1}\tilde y_1^{\overline b_0}.
$$
Then $\nu(x_1)>0$ and $\nu(\tilde y_1)=0$. Set
$$
R_1=R_0[x_1,\tilde y_1]_{m_{\nu}\cap R_0[x_1,\tilde y_1]}.
$$
By  \cite[Theorem 7.1]{CV},
$$
x_1, y_1=\frac{P_2}{x_1^{a_1\overline p_1^2}}
$$
are regular parameters in $R_1$ and
$$
P_{0,1}=x_1, P_{1,1}=\frac{P_2}{P_{0,1}^{a_1\overline p_1^2}},
$$
$$
P_{i,1}=\frac{P_{i+1}}{P_{0,1}^{a_1\overline p_1^2\cdots \overline p_i^2}}
$$
for $i\ge 1$ is a generating sequence for $\nu$ in $R_1$. For $i\ge 1$, we have
$$
P_{i+1,1}=\frac{P_{i+2}}{P_{0,1}^{a_1\overline p_1^2\cdots \overline p_{i+1}^2}}
=P_{i,1}^{\overline p_{i+1}^2}-(1+t)\tilde y_1^{\overline a_0a_{i+1}\overline p_{i+1}}x_1^{\overline p_{i+1}(\overline p_1a_{i+1}-a_1\overline p_1^2\cdots \overline p_{i}^2\overline p_{i+1})}.
$$
Set 
$$
a_{i,1}=\overline p_1a_{i+1}-a_1\overline p_1^2\cdots\overline p_i^2 \overline p_{i+1}
$$
and 
$$
\tau_{i,1}=\tilde y_1^{\overline a_0a_{i+1}},
$$
a unit in $R_1$ for $i\ge 1$. Then we have expressions
\begin{equation}\label{eq21}
P_{i+1,1}=P_{i,1}^{\overline p_{i+1}^2}-(1+t)\tau_{i,1}^{\overline p_{i+1}}x_1^{\overline p_{i+1}a_{i,1}}
\end{equation}
for $i\ge 1$. We have 
$$
\nu(x_1)=\frac{1}{\overline p_1}\nu(x)=\frac{1}{\overline p_1}\mbox{ and }\nu(P_{i,1})=\frac{a_{i,1}}{\overline p_{i+1}}\nu(x_1)
$$
with $\mbox{gcd}(a_{i,1},\overline p_{i+1})=1$ for all $i\ge 1$.

We have that 
$$
R_1/m_{R_1}=R_0/m_{R_0}\left[\frac{P_1^{\overline p_2}}{P_0^{a_1}}\right]=R_0/m_{R_0}[\alpha_1]=k(t)(\alpha_1),
$$
$$
\overline n_{i,1}:=[G(\nu(P_{0,1}),\ldots, \nu(P_{i,1})):G(\nu(P_{0,1}),\ldots,\nu(P_{i-1,1}))]=\overline n_{i+1}=\overline p_{i+1}
$$
for $i\ge 1$ and setting $U_{i,1}=x_1^{a_{i,1}}$ for $i\ge 1$ and
$$
\alpha_{i,1}=\left[\frac{P_{i,1}^{\overline p_{i+1}}}{U_{i,1}}\right]\in V_{\nu}/m_{\nu}.
$$
Let $\overline \tau_{i,1}\in R_1/m_{R_1}=k(t)(\alpha_1)$ be the class of $\tau_{i,1}$ in $V_{\nu}/m_{\nu}$.
We have that
$$
f_{i,1}(u)=u^{\overline p_{i+1}}-(1+t)\overline \tau_{i,1}^{\overline p_{i+1}}
$$
is the minimal polynomial of $\alpha_{i,1}$ over
$$
R_1/m_{R_1})(\alpha_{1,1},\ldots,\alpha_{i-1,1})=k(t)(\alpha_1,\ldots,\alpha_i).
$$
Thus
$$
d_{i,1}:= [R_1/m_{R_1}(\alpha_{1,1},\ldots,\alpha_{i,1}):R_1/m_{R_1}(\alpha_{1,1},\ldots,\alpha_{i-1,1})]=\overline p_{i+1}
$$
and
$$
n_{i,1}:= d_{i,1}\overline n_{i,1}=\overline p_{i+1}^2
$$
for $i\ge 1$.

Iterating this construction, we have an infinite sequence of birational extensions of regular local rings 
$$
R_0\rightarrow R_1\rightarrow \cdots\rightarrow R_j\rightarrow \cdots
$$
which are dominated by $\nu$ where $R_j$ has regular parameters $x_j,y_j$ and a generating sequence $\{P_{i,j}\}$ for $\nu$ defined by 
$$
P_{0,j}=x_j, P_{1,j}=y_j
$$
and for $i\ge 1$,
$$
P_{i+1,j}=P_{i,j}^{\overline p_{i+j}^2}-(1+t)\tau_{i,j}^{\overline p_{i+j}}x_j^{\overline p_{i+j}a_{i,j}}
$$
where $\tau_{i,j}$ are units in $R_j$ and $\mbox{gcd}(a_{i,j},\overline p_{i+j})=1$.  We have
$$
\nu(x_j)=\frac{1}{\overline p_1\cdots\overline p_j}\mbox{ and }\nu(P_{i,j})=\frac{a_{i,j}}{\overline p_{i+j}}\nu(x_j)
$$
for $i\ge 1$.

We have that 
$$
R_j/m_{R_j}=R_{j-1}/m_{R_{j-1}}\left[\frac{P_{1,j-1}^{\overline p_{1+j}}}{P_{0,j-1}^{a_{1,j-1}}}\right]=R_{j-1}/m_{R_{j-1}}[\alpha_{1,j-1}]=k(t)(\alpha_1,\ldots,\alpha_j),
$$
$$
\overline n_{i,j}:=[G(\nu(P_{0,j}),\ldots, \nu(P_{i,j})):G(\nu(P_{0,j}),\ldots,\nu(P_{i-1,j}))]=\overline n_{i+j}=\overline p_{i+j}
$$
for $i\ge 1$ and setting $U_{i,j}=x_j^{a_{i,j}}$ for $i\ge 1$ and
$$
\alpha_{i,j}=\left[\frac{P_{i,j}^{\overline p_{i+j}}}{U_{i,j}}\right]\in V_{\nu}/m_{\nu}.
$$
Let  $\overline \tau_{i,j}\in R_j/m_{R_j}=k(t)(\alpha_1,\ldots,\alpha_j)$ be the class of $\tau_{i,j}$ in $V_{\nu}/m_{\nu}$.
We have 
$$
f_{i,j}(u)=u^{\overline p_{i+j}}-(1+t)\overline \tau_{i,j}^{\overline p_{i+j}}
$$
 is the minimal polynomial of $\alpha_{i,j}$ over
$$
R_j/m_{R_j}(\alpha_{1,j},\ldots,\alpha_{i-1,j})=k(t)(\alpha_1,\ldots,\alpha_{i-1+j}).
$$
Thus
$$
d_{i,j}:= [R/m_{R_j}(\alpha_{1,j},\ldots,\alpha_{i,j}):R_j/m_{R_j}(\alpha_{1,j},\ldots,\alpha_{i-1,j})]=\overline p_{i+j}
$$
and
$$
n_{i,j}:= d_{i,j}\overline n_{i,j}=\overline p_{i+j}^2
$$
for $i\ge 1$.

Suppose $A$ is a regular local ring of the quotient field $K$ of $R_0$ which is dominated by $\nu$ and  dominates $R_0$. Then there exists a largest $l$ such that $A$ dominates $R_l$, so  there exist regular parameters $z$ and $w$ in $A$ such that $x_l$ and $y_l$ are monomials in $z$ and $w$. We thus have a factorization  (by  \cite[Theorem 3]{Ab2}) 
\begin{equation}\label{eq11}
R_l=D_0\rightarrow D_1\rightarrow \cdots\rightarrow D_n=A
\end{equation}
where $D_0$ has regular parameters $z=x_l,w=y_l$ and $D_{i+1}$ has regular parameters $z_{i+1}$, $w_{i+1}$, such that either $z_i=z_{i+1}$, $w_i=z_{i+1}w_{i+1}$, or
$z_i=z_{i+1}w_{i+1}$, $w_i=w_{i+1}$. 

We have shown that $D_0$ has a generating sequence $\{Q_i\}$ with $Q_i=P_{i.l}$ for $i\ge 0$, so that
$$
Q_0=z_0,\,\,Q_1=w_0,
$$
and 
$$
Q_{i+1}=Q_i^{\overline p_{i+l}^2}-(1+t)\tau_{i,l}^{\overline p_{i+l}}z_0^{\overline p_{i+l}e_{i,0}}
$$
where $e_{i,0}=a_{i,l}$ for $i\ge 1$.
We have $D_0/m_{D_0}=k(t)(\alpha_1,\ldots,\alpha_l)$.
Set
$$
\gamma_i=\alpha_{i,l}=\left[\frac{Q_i^{\overline p_{i+l}}}{Q_0^{e_{i,1}}}\right]
$$
for $i\ge 1$. We have that
\begin{equation}\label{eq40}
f_{i,l}(u)=u^{\overline p_{i+l}}-(1+t)\overline \tau_{i,l}^{\overline p_{i+l}}
\end{equation}
is the minimal polynomial of $\gamma_i$ over 
$$
D_0/m_{D_0}(\gamma_1,\ldots,\gamma_{i-1})=k(t)(\alpha_1,\ldots, \alpha_{i-1+l})
$$
so
$$
[D_0/m_{D_0}(\gamma_1,\ldots,\gamma_i):D_0/m_{D_0}(\gamma_1,\ldots,\gamma_{i-1})]=\overline p_{i+l}.
$$

We will show by induction on $j$ that $D_j$ has a generating sequence $\{Q_{i,j}\}$ defined by
\begin{equation}\label{eq12}
Q_{0,j}=z_j,\,\,Q_{1j}=w_j,
\end{equation}
\begin{equation}\label{eq13}
Q_{2,j}=Q_{1,j}^{\overline p_{1+l}c_j}-(1+t)\tau_{1,l}^{\overline p_{1+l}}Q_{0,j}^{\overline p_{1+l}e_{1,j}}
\end{equation}
with $c_j=[G(\nu(Q_{0,j},\nu(Q_{1,j})):G(\nu(Q_{1,j}))]$ and $\nu(Q_{1,j}^{c_j})=\nu(Q_{0,j}^{e_{1,j}})$ with $\mbox{gcd}(c_j,e_{1,j})=1$ and for $i\ge 2$,
\begin{equation}\label{eq14}
Q_{i+1,j}=Q_{i,j}^{\overline p_{i+l}^2}-(1+t)\tau_{i,l}^{\overline p_{i+l}}
Q_{0,j}^{\overline p_{i+l}e_{i,j}}Q_{1,j}^{\overline p_{i+l}f_{i,j}}
\end{equation}
with $\nu(Q_{i,j}^{\overline p_{i+l}})=\nu(Q_{0,j}^{e_{i,j}}Q_{1,j}^{f_{i,j}})$, and with
$$
\overline p_{i+j}=[G(\nu(Q_{0,j}),\ldots,\nu(Q_{i,j})):G(\nu(Q_{0,j}),\ldots,\nu(Q_{i-1,j}))].
$$
Further,
$$
\nu(Q_{2,j})>\overline p_{1+l}c_j\nu(Q_{1,j})=\overline p_{1+l}e_{1,j}\nu(Q_{0,j})
$$
and
$$
\nu(Q_{s,j})>\overline p_{s+l-1}^2\nu(Q_{s-1,j})
$$
for $s>2$.

We also have that $D_j/m_{D_j}=k(t)(\alpha_1,\ldots,\alpha_l)$ and
\begin{equation}\label{eq30}
\gamma_1=\left[\frac{Q_{1,j}^{c_j}}{Q_{0,j}^{e_{1,j}}}\right]\mbox{ and }
\gamma_i=\left[\frac{Q_{i,j}^{\overline p_{i+l}}}{Q_{0,j}^{e_{i,j}}Q_{1,j}^{f_{i,j}}}\right]\in V_{\nu}/m_{\nu}
\end{equation}
for $i\ge 2$.

We inductively construct the generating sequence $Q_{0,j},Q_{1,j},\ldots$ as follows. Suppose that $Q_{0,j},Q_{1,j},\ldots$ has been constructed. We will construct $Q_{0,j+1},Q_{1,j+1},\ldots$. We either have that 
\begin{equation}\label{eq1}
z_j=z_{j+1}w_{j+1}, w_j=w_{j+1}
\end{equation}
or
\begin{equation}\label{eq2}
z_j=z_{j+1}, w_j=z_{j+1}w_{j+1}.
\end{equation}

Suppose  (\ref{eq1}) holds, so that $\nu(w_j)<\nu(z_j)$. Then
\begin{equation}\label{eq7}
w_{j+1}=Q_{1,j+1}=Q_{1,j}\mbox{ and }z_{j+1}=Q_{0,j+1}=\frac{Q_{0,j}}{Q_{1,j+1}}.
\end{equation}
Substitute for $Q_{0,j}$, $Q_{1,j}$ in 
$$
Q_{2,j}=Q_{1,j}^{\overline p_{1+l}c_j}-(1+t)\tau_{1,l}^{\overline p_{1+l}}Q_{0,j}^{e_{1,j}\overline p_{1+l}},
$$
 to obtain
$$
Q_{2,j}=Q_{1,j+1}^{\overline p_{1+l}c_j}-(1+t)\tau_{1,l}^{\overline p_{1+l}}Q_{0,j+1}^{e_{1,j}\overline p_{1+l}}Q_{1,j+1}^{\overline p_{1+l}e_{1,j}}.
$$
Since $\nu(Q_{0,j})>\nu(Q_{1,j})$, and $\overline p_{1+l}c_j\nu(Q_{1,j})=\overline p_{1+l}e_{1,j}\nu(Q_{0,j})$, we have $\overline p_{1+l}e_{1,j}<\overline p_{1+l}c_j$. Thus a strict transform of $Q_{2,j}$ in $D_{j+1}$ is
\begin{equation}\label{eq8}
Q_{2,j+1}=\frac{Q_{2,j}}{Q_{1,j+1}^{\overline p_{1+l}e_{1,j}}}
=Q_{1,j+1}^{\overline p_{1+l}(c_j-e_{1,j})}-(1+t)\tau_{1,l}^{\overline p_{1+l}}Q_{0,j+1}^{\overline p_{1+l}e_{1,j}}.
\end{equation}

Suppose we have constructed the generating sequence out to 
$$
Q_{i,j+1}=\frac{Q_{i,j}}{Q_{1,j+1}^{e_{1,j}\overline p_{1+l}\overline p_{l+2}^2\cdots \overline p_{l+i-1}^2}}.
$$
Substituting into 
$$Q_{i+1,j}=Q_{i,j}^{\overline p_{i+l}^2}-(1+t)\tau_{i,l}^{\overline p_{i+l}}Q_{0,j}^{\overline p_{i+l}e_{i,j}}Q_{1,j}^{\overline p_{i+l}f_{i,j}},
$$
we have
$$
Q_{i+1,j}=Q_{1,j+1}^{e_{1,j}\overline p_{1+l}\overline p_{l+2}^2\cdots \overline p_{l+i}^2}Q_{i,j+1}^{\overline p_{i+l}^2}
-(1+t)\tau_{i,l}^{\overline p_{i+l}}Q_{0,j+1}^{\overline p_{i+l}e_{i,j}}Q_{1,j+1}^{\overline p_{i+l}(e_{i,j}+f_{i,j})}.
$$
We have that
$$
\overline p_{i+l}^2\nu(Q_{i,j})=\overline p_{i+l}e_{i,j}\nu(Q_{0,j})+\overline p_{i+l}f_{i,j}\nu(Q_{1,j})<\overline p_{i+l}(e_{i,j}+f_{i,j})\nu(Q_{0,j})
$$
since $\nu(Q_{1,j})<\nu(Q_{0,j})$, and from the inequalities
$$
\nu(Q_{s,j})>\overline p_{s+l-1}^2\nu(Q_{s-1,j})
$$
for $s>2$ and
$$
\nu(Q_{2,j})>\overline p_{1+l}e_{1,j}\nu(Q_{0,j})
$$
we have
$$
e_{1,j}\overline p_{1+l}\overline p_{2+l}^2\cdots \overline p_{i+l}^2<\overline p_{i+l}(e_{i,j}+f_{i,j})
$$
and thus
\begin{equation}\label{eq9}
Q_{i+1,j+1}=\frac{Q_{i+1,j}}{Q_{1,j+1}^{e_{1,j}\overline p_{1+l}\overline p_{2+l}^2\cdots \overline p_{i+l}^2}}
=Q_{i,j+1}^{\overline p_{i+l}^2}-(1+t)\tau_{i,l}^{\overline p_{i+l}}Q_{0,j+1}^{\overline p_{i+l}e_{i,j}}
Q_{1,j+1}^{\overline p_{i+l}(e_{i,j}+f_{i,j}-e_{1,j}\overline p_{1+l}\overline p_{2+l}^2\cdots \overline p_{i-1+l}^2\overline p_{i+l})}
\end{equation}
is a strict transform of $Q_{i+1,j}$ in $D_{j+1}$. 

Now suppose  (\ref{eq2}) holds, so that $\nu(w_j)>\nu(z_j)$. Then
\begin{equation}\label{eq6}
z_{j+1}=Q_{0,j+1}=Q_{0,j}\mbox{ and }w_{j+1}=Q_{1,j+1}=\frac{Q_{1,j}}{Q_{0,j+1}}.
\end{equation}
Substitute for $Q_{0,j}$ and $Q_{1,j}$ in
$$
Q_{2,j}=Q_{1,j}^{\overline p_{1,l}c_j}-(1+t)\tau_{1,l}^{\overline p_{1+l}}Q_{0,j}^{\overline p_{1+l}e_{1,j}}
$$
 to obtain
$$
Q_{2,j}=Q_{0,j+1}^{\overline p_{1+l}c_j}Q_{1,j+1}^{\overline p_{1+l}c_j}-(1+t)\tau_{1,l}^{\overline p_{1+l}}Q_{0,j+1}^{\overline p_{1+l}e_{1,j}}.
$$
Since $\nu(Q_{1,j})>\nu(Q_{0,j})$ and $\overline p_{1+l}c_j\nu(Q_{1,j})=\overline p_{1+l}e_{1,j}\nu(Q_{0,j})<\overline p_{1+l}e_{1,j}\nu(Q_{1,j})$
we have $\overline p_{1+l}c_j<\overline p_{1+l}e_{1,j}$. Thus a strict transform of $Q_{2,j}$ in $D_{j+1}$ is
\begin{equation}\label{eq3}
Q_{2,j+1}=\frac{Q_{2,j}}{Q_{0,j+1}^{c_j\overline p_{1+l}}}
=Q_{1,j+1}^{\overline p_{1+l}c_j}-(1+t)\tau_{1,l}^{\overline p_{1+l}}Q_{0,j+1}^{\overline p_{1+l}(e_{1,j}-c_j)}.
\end{equation}

Now suppose we have constructed the generating sequence out to 
$$
Q_{i,j+1}=\frac{Q_{i,j}}{Q_{0,j+1}^{c_j\overline p_{1+l}\overline p_{2+l}^2\cdots\overline p_{i-1+l}^2}}.
$$
Then substituting into 
$$
Q_{i+1,j}=Q_{i,j}^{\overline p_{i+l}^2}-(1+t)\tau_{i,l}^{\overline p_{i+l}}Q_{0,j}^{\overline p_{i+l}e_{i,j}}Q_{1,j}^{\overline p_{i+l}f_{i,j}},
$$
we have that
$$
Q_{i+1,j}=Q_{i,j+1}^{\overline p_{i+l}^2}Q_{0,j+1}^{c_j\overline p_{1+l}\overline p_{2+l}^2\cdots\overline p_{i+l}^2}
-(1+t)\tau_{i,l}^{\overline p_{i+l}}Q_{0,j+1}^{\overline p_{i+l}(e_{i,j}+f_{i,j})}Q_{1,j+1}^{\overline p_{i+l}f_{i,j}}.
$$
We have
$$
\overline p_{i+l}^2\nu(Q_{i,j})=\overline p_{i+l}e_{i,j}\nu(Q_{0,j})+\overline p_{i+l}f_{i,j}\nu(Q_{1,j})<\overline p_{i+l}(e_{i,j}+f_{i,j})\nu(Q_{1,j})
$$
since $\nu(Q_{0,j})<\nu(Q_{1,j})$, and from the inequalities
$$
\nu(Q_{s,j})>\overline p_{s-1+l}^2\nu(Q_{s-1,j})
$$
for $s>2$ and
$$
\nu(Q_{2,j})>\overline p_{1+l}c_j\nu(Q_{1,j}),
$$
we have 
$$
\overline p_{i+l}^2\overline p_{i-1+l}^2\cdots \overline p_{2+l}^2\overline p_{1+l}c_j<\overline p_{i+l}(e_{i,j}+f_{i,j})
$$
and thus
\begin{equation}\label{eq5}
Q_{i+1,j+1}=\frac{Q_{i+1,j}}{Q_{0,j+1}^{c_j\overline  p_{1+l}\overline p_{2+l}^2\cdots\overline p_{i+l}^2}}
=Q_{i,j+1}^{\overline p_{i+l}^2}-(1+t)\tau_{i,l}^{\overline p_{i+l}}Q_{0,j+1}^{\overline p_{i+l}(e_{i,j}+f_{i,j}-c_j\overline p_{1+l}\overline p_{2+l}^2\cdots \overline p_{i-1+l}^2\overline p_{i+l})}
Q_{1,j+1}^{\overline p_{i+l}f_{i,j}}
\end{equation}
is a strict transform of $Q_{i+1,j}$ in $D_{j+1}$.

Since $\mbox{gcd}(c_j,e_{1,j})=1$,
\begin{equation}\label{eq41}
[G(\nu(Q_{0,j}),\nu(Q_{1,j})):G(\nu(Q_{0,j}))]=c_j
\end{equation}
and since $G(\nu(Q_{0,j+1}),\nu(Q_{1,j+1}),\ldots,\nu(Q_{s,j+1}))=G(\nu(Q_{0,j}),\nu(Q_{1,j}),\ldots,\nu(Q_{s,j}))$ for $s\ge 1$, we have
\begin{equation}\label{eq42}
[G(\nu(Q_{0,j}),\nu(Q_{1,j}),\ldots,\nu(Q_{i,j})):G(\nu(Q_{0,j}),\nu(Q_{1,j}),\ldots,\nu(Q_{i-1,j}))]=\overline p_{i+l}
\end{equation}
for $i\ge 2$. Dividing the relation (\ref{eq13}) by $Q_{0,j}^{\overline p_{1+l}e_{1,j}}$ and taking the residue in $V_{\nu}/m_{\nu}$, we obtain
\begin{equation}\label{eq43}
0=\left[\frac{Q_{2,j}}{Q_{0,j}^{\overline p_{1+l}e_{1,j}}}\right]=f_{1,l}(\gamma_1)
\end{equation}
where $f_{1,l}$, defined by (\ref{eq40}), is the minimal polynomial of $\gamma_1$ over $D_j/m_{D_j}=k(t)(\alpha_1,\ldots,\alpha_l)$.
Dividing the relation (\ref{eq14}) by $(Q_{0,j}^{e_{i,j}}Q_{1,j}^{f_{i,j}})^{\overline p_{i+l}}$, and taking the residue in $V_{\nu}/m_{\nu}$, we obtain
\begin{equation}\label{eq44}
0=\left[\frac{Q_{i+1,j}}{(Q_{0,j}^{e_{i,j}}Q_{1,j}^{f_{i,j}})^{\overline p_{i+l}}}\right]=f_{i,l}(\gamma_i),
\end{equation}
where $f_{i,l}$, defined by (\ref{eq40}), is the minimal polynomial of $\gamma_i$ over $D_j/m_{D_j}(\gamma_1,\ldots,\gamma_{i-1})$ for $i\ge 2$.

We now verify (\ref{eq30}) by induction on $j$. If we are in case (\ref{eq1}), then the formula follows  for $j+1$ from induction and (\ref{eq7}), (\ref{eq8}) and (\ref{eq9}).
If we are in case (\ref{eq2}), then the formula follows for $j+1$ from induction  and (\ref{eq6}), (\ref{eq3}) and (\ref{eq5}).

The formula $D(i)$ of  \cite[Theorem 4.2]{CV} holds for the $Q_{0,j},Q_{1,j},\ldots$. That is, if we have a natural number $i$, a positive integer $m$ and natural numbers $f_n(s)$ for $1\le n\le i$ such that
$0\le f_1(s)<\overline p_{1+l}c_j$ for $1\le s\le m$ and $0\le f_n(s)<\overline p_{n+l}^2$ for $2\le s\le m$ and $1\le n\le i$. If
$$
\nu(Q_{0,j}^{f_0(s)}Q_{1,j}^{f_1(s)}\cdots Q_{i,j}^{f_i(s)})
= \nu(Q_{0,j}^{f_0(1)}Q_{1,j}^{f_1(1)}\cdots Q_{i,j}^{f_i(1)})
$$
for $1\le s\le m$, then
$$
1,\left[\frac{Q_{0.j}^{f_0(2)}Q_{1,j}^{f_1(2)}\cdots Q_{i,j}^{f_i(2)}}
                 {Q_{0.j}^{f_0(1)}Q_{1,j}^{f_1(1)}\cdots Q_{i,j}^{f_i(1)}}\right],\cdots,
                 \left[\frac{Q_{0.j}^{f_0(m)}Q_{1,j}^{f_1(m)}\cdots Q_{i,j}^{f_i(m)}}
                 {Q_{0.j}^{f_0(1)}Q_{1,j}^{f_1(1)}\cdots Q_{i,j}^{f_i(1)}}\right]
                 $$
are linearly independent over $D_j/m_{D_j}=k(t)(\alpha_1,\ldots,\alpha_l)$.
This formula follows from induction on $i$ and (\ref{eq30}), (\ref{eq40}), (\ref{eq41}) and (\ref{eq42})  as in the proof of $D(i)$ of  \cite[Theorem 4.2]{CV}.
Since $\nu$ has rank 1 by (\ref{eq60}),
 the fact that $Q_{0,j},Q_{1,j},\ldots$ is a generating sequence in $D_j$ is verified as in the proof of  \cite[Theorem 4.10 and Lemma 4.9]{CV}.

\section{Construction of unramified extensions which have larger valuation semigroups}

Let notation be as in the previous section. Let $K=k(t,x,y)$. Let $A$ be a regular local ring of $K$  which dominates $R_0$ and is dominated by $\nu$. Then there exists a factorization of $$
R_0\rightarrow R_l=D_0\rightarrow D_n=A
$$
 of the form  (\ref{eq11}). Let
\begin{equation}\label{eq51}
\mbox{$\lambda$ be a  $\overline p_{1+l}$-th root of $1+t$}
\end{equation}
in an algebraic closure of $K$,
$L=K(\lambda)$ and $\overline\nu$ be an extension of $\nu$ to $L$. Let $\omega\in k$ be a primitive $\overline p_{1+l}$-th root of unity. 

Let $\overline p=\overline p_{1+l}$ and $f(u)=u^{\overline p}-(1+t)$, the minimal polynomial of $\lambda$ over $K$. Let $B=A[\lambda]$. The ring $B$ is finite over $A$. We have the formula for the discriminant
$$
D_{L/K}(1,\lambda,\ldots,\lambda^{\overline p-1})=(-1)^{\frac{\overline p(\overline p-1)}{2}}\prod_{i=1}^{\overline p}\frac{df}{du}(\omega^i\lambda)
$$
by  \cite[Proposition 8.5 on page 204]{L} and  \cite[Formula (4) on page 204]{L}. Thus 
$$
D_{L/K}(1,\lambda,\ldots,\lambda^{\overline p-1})=(-1)^{\frac{\overline p(\overline p-1)}{2}}\overline p^{\overline p}(1+t)^{\overline p-1}.
$$
Now $1+t\in R_0$ and $1+t\not\in m_{R_0}$, so $1+t\not\in m_A$. Thus the discriminant ideal $D(B/A)$ in $A$ is equal to $A$, and so $B$ is normal and $A\rightarrow B$ is unramified by  \cite[Proposition 1.43]{RTM} and \cite[Theorem 1.44]{RTM}. Let $C=B_{m_{\overline\nu}\cap B}$. Then $A\rightarrow C$ is unramified, and so $C$ is a regular local ring, since the maximal ideal of $C$ is generated by a regular system of parameters in $A$.

\begin{Proposition}\label{Prop1} With the above notation, 
we have inequalities of semigroups
$$
S^C(\overline\nu)\ne S^A(\nu),
$$
under the natural inclusion $S^A(\nu)\subset S^C(\overline\nu)$.
\end{Proposition}

\begin{proof} Since $A=D_n$ in (\ref{eq11}), $A$ has regular parameters $z_n,w_n$ and a generating sequence
$$
Q_{0,n}=z_n, Q_{1,n}=w_n,Q _{2,n},\ldots
$$
defined by (\ref{eq12}), (\ref{eq13}) and (\ref{eq14}). so
$$
S^A(\nu)=S(\{\nu(Q_{i,n})\mid i\ge 0\}).
$$
We have 
\begin{equation}\label{eq45}
\nu(Q_{i+1,n})>\nu(Q_{i,n})
\end{equation}
for $i\ge 1$,
$$
[G(\nu(Q_{0,1}), \nu(Q_{1,n})):G(\nu(Q_{0,n})]=c_n
$$
 and
\begin{equation}\label{eq31}
[G(\nu(Q_{0,n}),\nu(Q_{1,n}),\ldots,\nu(Q_{i+1,n})):G(\nu(Q_{0,n}),\nu(Q_{1,n}),\ldots,\nu(Q_{i,n}))]=\overline p_{i+l+1}
\end{equation}
for $i\ge 1$.

By (\ref{eq13}), we have
$$
Q_{2,n}=Q_{1,n}^{\overline p_{1+l}c_n}-(1+t)\tau_{1,l}^{\overline p_{1+l}}Q_{0,1}^{\overline p_{1+l}e_{1,n}}
=w_n^{\overline p_{1+l}c_n}-(1+t)\tau_{1,l}^{\overline p_{1+l}}z_n^{\overline p_{1+l}e_{1,n}}.
$$
We have
$$
\gamma_1=\left[\frac{w_n^{c_n}}{z_n^{e_{1,n}}}\right]\in V_{\nu}/m_{\nu}\subset V_{\overline \nu}/m_{\overline \nu}.
$$
Let $\beta=[\lambda\tau_{i,l}]\in V_{\overline \nu}/m_{\overline \nu}$ which is nonzero since $\lambda$ and $\tau_{1,l}$ are units in $C$. There exists at most one index $j$ with $1\le j\le \overline p_{1+l}$ such that $\omega^j\beta=\gamma_1$. We have that 
$$
h_j=w_n^{\overline p_{1+l}}-\omega^j\lambda\tau_{1,j}z_n^{e_{1,n}}\in C
$$
for all $j$. If $\omega^j\beta\ne\gamma_1$, then $\overline \nu(h_j)=e_{1,n}\nu(z_n)$.
Since
$$
\sum_{j=1}^{\overline p_{1+l}}\overline \nu(h_j)=\nu(Q_{2,n})>\overline p_{1+l}e_{1,n}\nu(z_n),
$$
there exists a unique value of $j$ such that $\omega^j\beta=\gamma_1$, and $\overline\nu(h_j)>e_{1,n}\nu(z_n)$.
If $\overline \nu(h_j)\in S^A(\nu)$, we must then have that $\overline \nu(h_j)\in S(\nu(z_n),\nu(w_n))$, since 
$$
\overline\nu(h_j)=\nu(Q_{2,n})-(\overline p_{1+l}-1)e_{1,n}\nu(z_n)<\nu(Q_{2,n})
$$
 and by (\ref{eq45}).
Thus  $\nu(Q_{2,n})\in G(\nu(z_n),\nu(w_n))$, which is a contradiction to (\ref{eq31}).
\end{proof}

Let $\mu$ be a valuation of $V_{\nu}/m_{\nu}=k(t)[\{(1+t)^{\frac{1}{\overline p_i}}\}_{i\in \ZZ_+}]$ which is an extension of the $(t)$-adic valuation on $k[t]_{(t)}$. The value group of $\mu$ is $\ZZ$. Let $\phi$ be the composite valuation of $\nu$ and $\mu$ on $K$, so that the valuation ring of $\phi$ is $V_{\phi}=\pi^{-1}(V_{\mu})$ where $\pi:V_{\nu}\rightarrow V_{\nu}/m_{\nu}$ is the residue map (\cite[Section 10]{RTM}). The residue field of $\phi$ is $V_{\phi}/m_{\phi}=V_{\mu}/m_{\mu}\cong k$. Let $T_0=k[t,x,y]_{(t,x,y)}$, which is dominated by $\phi$.

\begin{Proposition}\label{Prop2} Suppose that $T$ is a  regular local ring of $K$ which dominates $T_0$ and is dominated by $\phi$. Then there exists a finite separable extension field $L$ of $K$ such that $T$ is unramified in $L$. Further, if $\overline\phi$ is an extension of $\phi$ to $L$, and if $U$ is the normal local ring of $L$ which lies over $T$ and is dominated by $\overline\phi$, then the following properties hold:
\begin{enumerate}
\item[1)] $U$ is a regular local ring
\item[2)]   the extension $T\rightarrow U$ is unramified with no residue field extension
\item[3)]  $S^U(\overline\phi)\ne S^T(\phi)$ under the natural inclusion $S^T(\phi)\subset S^{U}(\overline\phi)$.
\end{enumerate}
\end{Proposition}

\begin{proof} Let $R_0=(T_0)_{m_{\nu}\cap T_0}=k(t)[x,y]$ and $A=T_{m_{\nu}\cap T}$. By consideration of the factorization $R_0\rightarrow R_l\rightarrow A=D_n$ in (\ref{eq11}), 
let $\lambda$ be a $\overline p_{1+l}$-th root of $1+t$ in an extension field of $K$, and let $L=K(\lambda)$ (as in (\ref{eq51})). Let $\overline\phi$ be an extension of $\phi$ to $L$ and let $\overline\nu$ be the extension of $\nu$ to $L$ with which $\overline\phi$ is composite.

Let $U$ be the normal local ring of $L$ which is dominated by $\overline\phi$ and lies over $T$. Then  $U$ is a regular local ring and the extension $T\rightarrow U$ is unramified with no residue field extension, by the argument before Proposition \ref{Prop1} and since $V_{\overline\phi}/m_{\overline\phi}=V_{\phi}/m_{\phi}=k$ by  \cite[Corollary 2, page 26]{ZS2}. Let $C=U_{m_{\overline \nu}\cap U}$. Then  $C$ is a  regular local ring and $A\rightarrow C$ is unramified (by the argument before Proposition \ref{Prop1}). We have that
\begin{equation}\label{eq52}
S^C(\overline\nu)\ne S^A(\nu)
\end{equation}
 by Proposition \ref{Prop1}.

By the explanation on  \cite[page 56]{RTM} or  \cite[Theorem 17, page 43]{ZS2}, we have a commutative diagram of homomorphisms of value groups,
where the horizontal sequences are short exact and the vertical arrows are injective,
$$
\begin{array}{ccccccccc}
0&\rightarrow&\Phi_{\mu}&\rightarrow&\Phi_{\phi}&\rightarrow & \Phi_{\nu}&\rightarrow &0\\
&&\downarrow&&\downarrow&&\downarrow&&\\
0&\rightarrow&\Phi_{\overline\mu}&\rightarrow&\Phi_{\overline\phi}&\rightarrow & \Phi_{\overline\nu}&\rightarrow &0,
\end{array}
$$
which induce a commutative diagram of homorphisms of semigroups, where the horizontal arrows are surjective and the vertical are injective,
$$
\begin{array}{ccc}
S^T(\phi)&\rightarrow&S^A(\nu)\\
\downarrow&&\downarrow\\
S^U(\overline\phi)&\rightarrow&S^C(\overline\nu).
\end{array}
$$
We will show that the first horizontal arrow is surjective. The proof for the second horizontal arrow is the same. Suppose $\alpha\in S^A(\nu)$. Then there exists $f\in A$ such that $\nu(f)=\alpha$. There exists $g\in T\setminus (x,y)$ such that $gf\in T$. Thus $\phi(g)\in \Phi_{\mu}$ and $\phi(gf)=\phi(g)+\phi(f)$ so  $\phi(gf)\in S^T(\phi)$ maps onto $\nu(f)=\alpha$.
Thus $S^T(\phi)\ne S^U(\overline\phi)$ by (\ref{eq52}).
\end{proof}

\section{Proofs of Theorems \ref{Theorem1} and \ref{Theorem2} }\label{SecCounter}

We first give the proof of Theorem \ref{Theorem1}. 

A Henselization $T^h$ of $T$ can be constructed as follows, as is explained in  \cite[Chapter VII]{N}. Let $N$ be a separable closure of $K$. Then $N$ is an (infinite)  Galois extension of $K$ with Galois group $G(N/K)$. Let $E$ be a local ring of the integral closure of $T$ in $N$, and let
$$
G^s(E/T)=\{\sigma\in G(N/K)\mid \sigma(E)=E\}.
$$
A Henselization $T^h$ of $T$ is then
$T^h=E^{G^s(E/T)}$, which is a local ring of the fixed field $M=N^{G^s(E/T)}$ of $G^s(E/T)$ which lies over $T$.

Let $K\rightarrow L$ be the field extension of Proposition \ref{Prop2}. Choose an embedding $K\rightarrow L\rightarrow N$ of $L$ as a subfield of $N$, and let $U$ be the local ring of the integral closure of $T$ in $L$ which is dominated by $E$. By Proposition \ref{Prop2}, $U$ is unramified over $T$ with no residue field extension. Thus $G^s(E/T)\subset G(N/L)$ (c.f.  \cite[Section 2]{Ab2} or   \cite[Section 4]{C1}). Thus we have that 
$$
L=N^{G(N/L)}\subset N^{G^s(E/T)}=M.
$$
Thus $U$ is dominated by $T^h$ since $U=E\cap L$ and $T^h=E\cap M$. Let $\overline\phi=\phi^h|L$. Then $\overline \phi$ dominates $U$ and $T^h$ dominates $U$, so $S^U(\overline\phi)\subset S^{T^h}(\phi^h)$. But $S^U(\overline\phi)\ne S^T(\phi)$ by Proposition \ref{Prop2}, so $S^{T^h}(\phi^h)\ne S^T(\phi)$.
\vskip .2truein
We now give the proof of Theorem \ref{Theorem2}.

 By  \cite[Theorem 43.5]{N}, the completion $\hat T$ of $T$ is a Henselian local ring, and so by  \cite[Theorem 30.3]{N}, $\hat T$ dominates the Henselization $T^h$ of $T$. 

Now $T^h=F_m$ where $F$ is the integral closure of $T$ in the quotient field $M$ of $T^h$ and $m$ is a maximal ideal of $F$. Suppose $q$ is a nonzero prime ideal of $T^h$. Then there exists a nonzero element $f\in q$, so that $f=\frac{g}{h}$ where $g,h\in F$ and $h\not\in m$. Now $g$ is integral over $T$ and $T$ is  normal, so the norm $N_{K(g)/K}(g)$ satisfies
$0\ne N_{K(g)/K}(g)\in q\cap T$ by  \cite[Theorem 4, page 260]{ZS1} and  \cite[formula (15) on page 91]{ZS1}. Thus $q\cap T\ne (0)$. 

Suppose there exists a prime ideal $p$ in $\hat T$ with an extension $\hat\phi$ of $\phi$ to the quotient field of $\hat T/p$ which dominates $\hat T/p$ such that $S^{\hat T/p}(\hat\phi)=S^T(\phi)$.
Then $p\cap T^h=(0)$, and so $\hat T/p$ dominates $T^h$. Let $\phi^h$ be the restriction of $\hat\phi$ to the quotient field of $T^h$. We then have natural inclusions
$$
S^T(\phi)\subset S^{T^h}(\phi^h)\subset S^{\hat T/p}(\hat\phi).
$$
But $S^T(\phi)\ne S^{T^h}(\phi^h)$ by Theorem \ref{Theorem1}, giving a contradiction to our assumption that $S^{\hat T/p}(\hat\phi)=S^T(\phi)$.


\begin{thebibliography}{1000000000}
\bibitem{Ab1} S. Abhyankar, On the valuations centered in a local domain, Amer. J. Math. 78 (1956), 321 - 348.
\bibitem{Ab2} S. Abhyankar, Local uniformization on algebraic surfaces over ground fields of characteristic $p\ne 0$, Annals Math. 63 (1956), 491 - 526.
\bibitem{RTM} S. Abhyankar, Ramification theoretic methods in algebraic geometry, Princeton Univ. Press, 1959.
\bibitem{C2}  S.D. Cutkosky, Local factorization and monomialization of morphisms, Ast\'erisque
260, 1999.\bibitem{C1} S.D. Cutkosky, Finite generation of extensions of associated graded rings along a valuation, to appear in the Journal of the London Math. Soc.
\bibitem{CE} S.D. Cutkosky and S. El Hitti, Formal prime ideals of infinite value and their algebraic resolution, Ann. Fact. Sci. Toulouse Math. 19 (2010), 635 - 649.
\bibitem{CG} S.D. ~Cutkosky and L. ~Ghezzi,  \textit{Completions of valuation rings},
Contemp. Math. \textbf{386} (2005), 13 - 34.
\bibitem{CV} S.D. Cutkosky and Pham An Vinh, Valuation semigroups of two dimensional  local rings, Proceedings of the London Mathematical Society 108 (2014), 350 - 384.
\bibitem{Eh} S. El Hitti, Perron transforms, Comm. Algebra 42 (2014), 2004 - 2045.
\bibitem{HS} W. Heinzer and J. Sally, Extensions of valuations to the completion of a local domain, J. Pure and Appl. Algebra 71 (1991), 175 - 185.
\bibitem{GAST} F.J. Herrera Govantes, F.J. Olalla Acosta, M. Spivakovsky, G. Teissier, Extending a valuation centered in a local domain to its formal completion, Proc. London Math. Soc. 105 (2012), 571 - 621.
\bibitem{K} O. Kashcheyeva, Constructing examples of semigroups of valuations, J. Pure Appl. Algebra 200 (2016), 3826 - 3860.
\bibitem{L} S. Lang, Algebra, Revised third edition, Springer Verlag, 2002
\bibitem{Mo} M. Moghaddam, A construction for a class of valuations of the field $K(X_1,\ldots, X_d,Y)$ with large value group, Journal of Algebra, 319, 7 (2008), 2803-2829.
\bibitem{N} M. Nagata, Local Rings, Interscience publishers, New York, London, 1962.
\bibitem{NS} J. Novacoski and M. Spivakovsky, Key polynomials and pseudo-convergent sequences, J. Algebra 495 (2018), 199 - 219.
\bibitem{S} M. Spivakovsky, Valuations in function fields of surfaces, Amer. J. Math. 112 (1990), 107 - 156.
 \bibitem{T1} B. Teissier,   Valuations, deformations and toric geometry, Valuation theory and its applications II,
F.V. Kuhlmann, S. Kuhlmann and M. Marshall, editors, Fields
Institute Communications 33 (2003), Amer. Math. Soc., Providence, RI, 361
-- 459.
\bibitem{T2} B. Teissier, Overweight deformations of affine toric varieties and local uniformization, in Valuation theory in interaction, Proceedings of the second international conference on valuation theory, Segovia-El Escorial, 2011. Edited by A. Campillo, F-V- Kehlmann and B. Teissier. European Math. Soc. Publishing House, Congress Reports Series, Sept. 2014, 474 - 565.
\bibitem{ZS1} O. Zariski and P. Samuel, Commutative Algebra Volume I, Van Nostrand, 1958.
\bibitem{ZS2} O. Zariski and P. Samuel, Commutative Algebra Volume II, Van Nostrand, 1960.
\end{thebibliography}
\end{document}